\documentclass[11pt]{amsart}
\usepackage[T1]{fontenc}
\usepackage[utf8]{inputenc}

\usepackage{amsmath,amssymb,amsthm,amsfonts,mathtools,bm}

\usepackage{lmodern}
\usepackage[final]{microtype}
\usepackage[english]{babel}
\usepackage{nicefrac} % Nice inline fractions

\usepackage[margin=1.2in]{geometry}

\usepackage{multicol}
\usepackage{array}
\usepackage{booktabs}
\usepackage[shortlabels]{enumitem}

\usepackage{tikz}
\usetikzlibrary{cd}
\usetikzlibrary{shapes}
\usetikzlibrary{shapes.multipart}
\usetikzlibrary{math}

\usepackage{thmtools}
\usepackage{thm-restate}
\usepackage[hyphens]{url}
\theoremstyle{plain}
\newtheorem{theorem}{Theorem}[section]

\newtheorem*{claim*}{Claim}

\newtheorem*{proposition*}{Proposition}

\newtheorem*{fact*}{Fact}
\newtheorem{conjecture}[theorem]{Conjecture}
\newtheorem*{conjecture*}{Conjecture}
\newtheorem{corollary}[theorem]{Corollary}
\newtheorem{lemma}[theorem]{Lemma}
\newtheorem*{lemma*}{Lemma}

\newtheorem*{question*}{Question}
\theoremstyle{definition}\newtheorem{remark}[theorem]{Remark}
\theoremstyle{definition}\newtheorem*{remark*}{Remark}
\theoremstyle{definition}\newtheorem{definition}[theorem]{Definition}
\theoremstyle{definition}\newtheorem*{definition*}{Definition}
\theoremstyle{definition}
\theoremstyle{definition}
\theoremstyle{definition}
\theoremstyle{definition}\newtheorem*{example*}{Example}

\newcommand\floor[1]{\lfloor #1 \rfloor} % Floor
\newcommand\ceil[1]{\lceil #1 \rceil} % Ceiling
\renewcommand{\phi}{\varphi} % Nice looking phi
\renewcommand{\epsilon}{\varepsilon} % Nice looking epsilon
\def\N{\mathbb{N}} % Natural numbers
\def\R{\mathbb{R}} % Reals
\def\S{\mathcal{S}}

\title{Bounding the dimension of exceptional sets for orthogonal projections}

\author[P. Cholak]{Peter Cholak}
\address{University of Notre Dame}  
\email{Peter.Cholak.1@nd.edu}

\author[M. Cs\"ornyei]{Marianna Cs\"ornyei}
\address{University of Chicago}  
\email{csornyei@math.uchicago.edu}

\author[N. Lutz]{Neil Lutz}
\address{Swarthmore College}
\email{nlutz1@swarthmore.edu}

\author[P. Lutz]{Patrick Lutz}
\address{UC Berkeley}
\email{pglutz@berkeley.edu}

\author[E. Mayordomo]{Elvira Mayordomo}
\address{Universidad de Zaragoza}
\email{elvira@unizar.es}

\author[D. Stull]{Donald M. Stull}
\address{University of Chicago}
\email{dmstull@uchicago.edu}

\thanks{The project was begun at and was partially made possible by a SQuaRE at the American Institute of Mathematics. The authors thank AIM for providing a supportive and mathematically rich environment.  We also want to acknowledge support by Notre Dame Global via an International Research Travel Grant. We thank an anonymous reviewer for helpful comments.}

\begin{document}

\begin{abstract}
    It is well known that if $A\subseteq\R^n$ is an analytic set of Hausdorff dimension $a$, then $\dim_H(\pi_VA)=\min\{a,k\}$ for a.e.\ $V\in G(n,k)$, where $G(n,k)$ denotes the set of all $k$-dimensional subspaces of $\R^n$ and   $\pi_V$ is the orthogonal projection of $A$ onto $V$. In this paper we study how large the exceptional set
    \begin{equation*}
        \{V\in G(n,k) \mid \dim_H(\pi_V A) < s\}
    \end{equation*}
    can be for a given $s\le\min\{a,k\}.$ We improve previously known estimates on the dimension of the exceptional set, and we show that our estimates are sharp for $k=1$ and for $k=n-1$.  Hence we completely resolve this question for $n=3$. 
\end{abstract}

\maketitle

\section{Introduction}
Let $A\subseteq\R^2$ be an analytic set of Hausdorff dimension $a$. For any $e\in \mathcal{S}^1$, let $\pi_e A$ denote the orthogonal projection of $A$ onto the line through the origin in the direction of $e$. In~\cite{Marstrand54}, Marstrand proved that, for almost every $e \in \S^1$,
\begin{equation*}
    \dim_H(\pi_e A) = \min\{\dim_H(A), 1\}.
\end{equation*}
Given the almost everywhere nature of Marstrand's theorem, it is natural to investigate the size of the exceptional set of directions. That is, given $A\subseteq\R^2$ and $s \leq \min\{\dim_H(A), 1\}$, to bound the dimension of the exceptional set of $A$,
\begin{equation*}
    E_s(A) = \{e\in\mathcal{S}^1\mid \dim_H(\pi_e A) < s\}.
\end{equation*}
Kaufman \cite{Kaufman68} gave the first non-trivial upper bound, showing that $\dim_H(E_s(A)) \leq s$. Oberlin \cite{Oberlin12} conjectured that, for any analytic set $A$ of Hausdorff dimension $a$ and any $s \leq \min\{a, 1\}$, $\dim_H(E_s(A))\leq \max\{2s-a, 0\}$. In a recent breakthrough, Ren and Wang \cite{ren2023} proved Oberlin's conjecture. 
\begin{theorem}[\cite{ren2023}]\label{thm:RenWang}
    Let $A\subseteq\R^2$ be an analytic set, $a=\dim_H(A)$, and $0 < s \leq \min\{a,1\}$. Then
    \begin{equation*}
        \dim_H(\{e\in\mathcal{S}^1\mid \dim_H(\pi_e A) < s\}) \leq \max\{2s-a, 0\}.
    \end{equation*}        
\end{theorem}

In this paper, we consider the higher-dimensional analog of the exceptional set estimate. Let $G(n,k)$ denote the set of all $k$-dimensional subspaces of $\R^n$. For any $V\in G(n,k)$, let $\pi_V:\R^n \rightarrow V$ denote the orthogonal projection onto $V$. Mattila \cite{Mattila75} generalized Marstrand's theorem in this context, showing that for any analytic $A\subseteq\R^n$, $\dim_H(\pi_V A) = \min\{\dim_H(A), k\}$, for almost every $V \in G(n,k)$. 
 
Let $n, k\in \N$ and $a, s\in \R$ such that $k < n$, $a \leq n$ and $0 < s \leq \min\{a,k\}$. Define
\[
 T_{n, k}(a, s) = \sup_{A \subseteq \R^n, \,\dim_H(A) = a}\dim_H(\{V \in G(n, k) \mid \dim_H(\pi_V A) < s\}),
\]
where the supremum is taken over all analytic subsets. 

Mattila \cite{Mattila75} proved that
\begin{equation}\label{eq:MattilaBound}
    T_{n,k}(a,s) \leq k(n-k) + s - k.
\end{equation}
Note that Mattila's bound generalizes Kaufman's bound to higher dimensions. 

Mattila's bound was subsequently improved by several authors. Falconer \cite{Falconer82}, and later Peres and Schlag \cite{PerSch00}, proved that
\begin{equation}\label{eq:FalconerBound}
    T_{n,k}(a,s) \leq k(n-k) + s - a
\end{equation}
which improves Mattila's bound when $a > k$. Ren \cite{Ren23} gave an $\epsilon$-improvement of Mattila's bound for the case where $k = n-1$. He \cite{He20} showed that if $k = 1$ or $k = n - 1$ then $T_{n, k}(a, \frac{k}{n}a) \leq k(n - k) - 1$, which improves previous bounds for certain values of $a$. Finally, Gan \cite{Gan24} further improved these bounds and showed that his bounds are sharp for certain ranges of $a$ and $s$.\footnote{Gan's bounds are difficult to calculate due to their reliance on the Brascamp-Lieb inequality (which is expressed in terms of a $\sup$ over an $\inf$) and his paper only provides more explicit bounds for certain values of $a$ and $s$. Gan has shown that these explicit bounds are sharp when either $k \leq \frac{n}{2}$, $1 < a \leq 2$ and $s$ falls in a certain range depending on $a$ or $k \geq \frac{n}{2}$, $n - 1 < a \leq n$ and $s$ falls in a certain range depending on $a$. For more detail on Gan's bounds and when they are sharp, it is best to consult Gan's paper directly.} However, in general, sharp bounds for $T_{n,k}(a,s)$ are unknown. 

% In this paper, we will prove new lower and upper bounds on $T_{n, k}(a, s)$ which are given by the following function.

% The main goal of this paper is to prove sharp upper bounds on $T_{n, k}(a, s)$ in the cases where $k = 1$ and $k = n - 1$. To do so, we introduce a function $S_{n, k}(a, s)$ and show that it is always a lower bound on $T_{n, k}(a, s)$ (though it is not in general a tight lower bound). We then show that in the cases $k = 1$ and $k = n - 1$, $S_{n, k}(a, s)$ is also an upper bound on $T_{n, k}(a, s)$. We now state our results more formally.

% \begin{definition}
% For any $n, k \in \N$ and $a, s \in \R$, let $S_{n, k}(a, s)$ denote
% \[
%     S_{n, k}(a, s) = k(n-k-\floor{a - s}-1)+\max\{\ceil{s} - 1, \floor{a-s}+2s-a\}.
% \]
% \end{definition}

The main goal of this paper is to prove sharp upper bounds on $T_{n,k}(a,s)$ in the cases where $k=1$ and $k=n-1$, as well as to provide lower bounds on $T_{n, k}(a, s)$ for all values of $n$ and $k$, which we conjecture to be sharp. To do so, we introduce a function $S_{n, k}(a, s)$.
\begin{definition}
For any $n, k \in \N$ and $a, s \in \R$, let $S_{n, k}(a, s)$ denote
\[
    S_{n, k}(a, s) = k(n-k) -\left(\floor{a - s}+1\right)\left(k - \ceil{s} +1\right)+\max\{0, \floor{a-s}+2s-a - \ceil{s}+1\}.
\]
\end{definition}

\begin{remark}
Since it is somewhat difficult to understand the function $S_{n, k}(a, s)$ simply by staring at the formula, we have included Figure~\ref{fig:chart1}, which shows the values of $S_{n, k}(a, s)$ for the fairly representative case where $n = 10$ and $k = 3$.
\end{remark}

% So $S_{n, 1}(a, s) = (n-1) -\left(\floor{a - s}+1\right)+\max\{0, \floor{a-s}+2s-a\}$ and $S_{n, n-1}(a, s) = (n-1) -\left(\floor{a - s}+1\right)\left(n - \ceil{s}\right)+\max\{0, \floor{a-s}+2s-a - \ceil{s}+1\}.$ 

We show that $S_{n, k}(a, s)$ is always a lower bound on $T_{n, k}(a, s)$ and that when $k = 1$ or $k = n - 1$, they are equal (and thus in these cases we have found a sharp upper bound for $T_{n, k}(a, s)$).

\begin{theorem}\label{thm:maintheorem}
Let $n \geq 2$, $0 < a \leq n$, $1\leq k \le n-1$, and $\max\{0,a -n+k\} < s \leq \min\{a,k\}$. Then $T_{n, k}(a, s) \geq S_{n, k}(a, s)$. Furthermore, if $k=1$ or $k=n-1$, then $T_{n, k}(a, s) = S_{n, k}(a, s)$.
\end{theorem}

Although we cannot show that $T_{n, k}(a, s)$ is equal to $S_{n, k}(a, s)$ for all values of $n, k, a$ and $s$, we conjecture that it is.

\begin{conjecture}
\label{conj:main}
Let $n \geq 2$, $0 < a \leq n$, $1\leq k \le n-1$, and $\max\{0,a -n+k\} < s \leq \min\{a,k\}$. Then $T_{n, k}(a, s) = S_{n, k}(a, s)$.
\end{conjecture}

% We show that $S_{n,k}(a,s)$ is always a lower bound on $T_{n,k}(a,s)$. We conjecture that it actually gives sharp bounds for $T_{n,k}$.

% \begin{conjecture}
% \label{conj:main}
% For any $n, k \in \N$ and $a, s \in \R$, if $n \geq 2$, $1\leq k \leq n -1$, $0 < a \leq n$, and $\max\{0,a-n+k\} < s \leq \min\{a,k\}$, then
%  \begin{equation*}
%       T_{n,k}(a,s) = S_{n, k}(a, s).
%   \end{equation*}
% \end{conjecture}

% Although we cannot give a full proof of Conjecture \ref{conj:main}, we can prove that $S_{n,k}(a,s)$ gives sharp bounds for $T_{n,k}(a,s)$ in the cases where $k=1$ and $k=n-1$.

The restrictions on $n, k, a$ and $s$ in the theorem and conjecture above are included in order to rule out trivial cases. In particular, note that for every $A\subseteq \R^n$ and $V\in G(n,k)$, $A\subseteq \pi_V A\times V^\perp$ and therefore $\dim_H(\pi_V A)\ge \dim_H(A)-(n-k)$. Since we always have $\dim_H(\pi_V A)\ge 0$, it follows that if $s \leq \max\{0,a -n+k\}$, then $T_{n,k}(a,s) = 0$. 

\begin{remark}
Proposition~37 of Gan~\cite{Gan24} gives a lower bound on $T_{n, k}(a, s)$. A careful comparison of our lower bounds with this proposition shows that the two lower bounds are equal in the regions that Gan calls Type I, II, and III while in Type IV our lower bounds are larger. In addition, we mentioned above that Gan finds sharp bounds for $T_{n, k}(a, s)$ for certain ranges of $a$ and $s$ and it is not difficult to verify that these bounds agree with our lower bounds in this range.
\end{remark}

\begin{remark}
\label{rmk:simplified}
In the cases for which we prove sharp upper bounds for $T_{n, k}(a, s)$ (namely when $k = 1$ and $k = n - 1$), we can slightly simplify the formula for $S_{n, k}(a, s)$. In particular, we have
\begin{align*}
S_{n, 1}(a, s) &= n - 2 - \floor{a - s} + \max\{0, \floor{a-s}+2s-a\}\\
S_{n, n-1}(a, s) &= \ceil{s} - 1 + \max\{0, 2s-a - \ceil{s}+1\}.\\
\end{align*}
In order to help understand these expressions, Figure~\ref{fig:chart2} shows the values of $S_{n, 1}(a, s)$ and $S_{n, n - 1}(a, s)$ when $n = 4$.
\end{remark}

% \begin{remark}
% It is somewhat difficult to understand the statement of our theorem simply by looking at the formula for $S_{n, k}(a, s)$. In order to help make clear what our theorem says, we have included the figure below, which shows the values of $S_{n, k}(a, s)$ for $n = 10$ and $k = 3$.
% \end{remark}

\begin{figure}
\label{fig:chart1}
\begin{tikzpicture}
% Draw the rectangle
\draw[thick] (0,0) rectangle (10, 3);
  
% Draw dotted vertical lines
\foreach \x in {1, ..., 10} {
  \draw[dotted] (\x,0) -- (\x,3);
}

% Draw dotted horizontal lines (corresponding to different values of l)
\foreach \x in {1, 2} {
  \draw[dotted] (0, \x) -- (10, \x);
  \draw (\x, \x) -- (7 + \x, \x);
}

% Draw slanted lines corresponding to different values of m
\foreach \x in {0, ..., 7} {
  \draw[thick] (\x, 0) -- (\x + 3 + 0.25, 3 + 0.25);
}

% Label regions with different values of m
\foreach \x in {1, ..., 7} {
  \draw (\x + 2.65, 3.2) node {\tiny $m = \x$};
}

\foreach \y in {0, 1, 2} {
  \foreach \x in {0, ..., 6}{
    \draw (\x + \y, \y) -- (\x + \y + 2, \y + 1);
    \draw[fill=gray, opacity=0.3] (\x + \y, \y) -- (\x + \y + 1, \y + 1) -- (\x + \y + 2, \y + 1) -- cycle;
    
    \newcount\a
    \tikzmath{\a = 21 - (\x + 1)*(3 - \y);}
    \draw (\x + \y + 0.8, \y + 0.2) node {\tiny $\the \a$};
  }
}

\draw (1, 2) node {N/A};  %was 21
\draw (9, 1) node {$0$ ($\emptyset$)};  %was 0

\draw (-0.75, 0.5) node {$\ceil{s} = 1$}; %$l = 1$};
\draw (-0.75, 1.5) node {$\ceil{s} = 2$};%{$l = 2$};
\draw (-0.75, 2.5) node {$\ceil{s} = 3$};%{$l = 3$};
\end{tikzpicture}
 \caption{The values of $S_{n, k}(a, s)$ for $n=10$ and $k=3$. The horizontal axis is $a$ and the vertical axis is $s$. The integer in each unshaded region denotes the value of $S_{n, k}(a, s)$ in that region; the value of $S_{n, k}(a, s)$ in each shaded region is given by $k(n - k - m) + (m - 1) \ceil{s}+ 2s - a$ where $m = \floor{a - s} + 1$.}
 \end{figure}

\begin{figure}
\centering
\begin{tikzpicture}
\draw[thick] (0, 1) rectangle (4, 2);
  
\foreach \x in {1, 2, 3} {
  \draw[dotted] (\x, 1) -- (\x, 2);
}

\foreach \x in {0, 1, 2, 3} {
    \draw[thick] (\x, 1) -- (\x + 1, 2);
}

\foreach \x in {0, 1, 2} {
    \draw[fill = gray, opacity=0.3] (\x, 1) -- (\x + 1, 2) -- (\x + 2, 2) -- cycle;
    \draw (\x, 1) -- (\x + 2, 2);
    \newcount\a
    \tikzmath{\a = 2 - \x;}
    \draw (\x + 0.8, 1.2) node {\tiny $\the \a$};
}

\draw (-.7, 1.9) node {$s = 1$};
\draw (4, .7) node {$a = 4$};

\draw[thick] (7,0) rectangle (11, 3);
  
\foreach \x in {1, 2, 3} {
  \draw[dotted] (\x + 7,0) -- (\x + 7,3);
}

\foreach \y in {1, 2} {
  \draw[dotted] (7,\y) -- (11,\y);
  \draw[thick] (7 + \y, \y) -- (8 + \y, \y);
}

\foreach \x in {0, 1} {
    \draw[thick] (\x + 7, 0) -- (\x + 10, 3);
}

\foreach \x in {0, 1, 2} {
    \draw[fill = gray, opacity=0.3] (\x + 7, \x) -- (\x + 8, \x + 1) -- (\x + 9, \x + 1) -- cycle;
    \draw (\x + 7, \x) -- (\x + 9, \x + 1);
    \newcount\a
    \draw (\x + 7.8, \x + 0.2) node {\tiny $\x$};
}

\draw (6.3,2.9) node {$s = 3$};
\draw (11, -.3) node {$a = 4$};

% \foreach \x in {0, 1, 2} {
%     \draw[fill = gray, opacity=0.3] (\x, 0) -- (\x + 1, 1) -- (\x + 2, 1) -- cycle;
%     \newcount\a
%     \tikzmath{\a = 2 - \x;}
%     \draw (\x + 0.8, 0.2) node {\tiny $\the \a$};
% }

% \foreach \y in {0} {
%   \foreach \x in {0, ..., 2}{
%     \draw (\x + \y, \y) -- (\x + \y + 2, \y + 1);
%     \draw[fill=gray, opacity=0.3] (\x + \y, \y) -- (\x + \y + 1, \y + 1) -- (\x + \y + 2, \y + 1) -- cycle;
    
%     \newcount\a
%     \tikzmath{\a = 21 - (\x + 1)*(3 - \y);}
%     \draw (\x + \y + 0.8, \y + 0.2) node {\tiny $\the \a$};
%   }
% }

% \draw (1, 2) node {N/A};  %was 21
% \draw (9, 1) node {$0$ ($\emptyset$)};  %was 0

% \draw (-0.75, 0.5) node {$\ceil{s} = 1$}; %$l = 1$};
% \draw (-0.75, 1.5) node {$\ceil{s} = 2$};%{$l = 2$};
% \draw (-0.75, 2.5) node {$\ceil{s} = 3$};%{$l = 3$};
\end{tikzpicture}
\caption{The values of $S_{n, k}(a, s)$ (and hence also $T_{n, k}(a, s)$) when $n = 4$ for $k = 1$ and $k = 3$. The horizontal axis is $a$ and the vertical axis is $s$. The integer in each unshaded region denotes the value of $S_{n, k}(a, s)$ in that region. The value of $S_{n, k}(a, s)$ in each shaded region is non-integral and is given by the formula in Remark~\ref{rmk:simplified}.}
\label{fig:chart2}
\end{figure}

\section{Proof of lower bound}\label{sec:lb}

In this section, we will prove the first part of Theorem~\ref{thm:maintheorem}, namely that for any $n, k, a$ and $s$ which satisfy the assumptions of the theorem, $T_{n,k}(a,s)\ge S_{n, k}(a,s)$. We prove this by induction on $n$, with $n = 2$ as the base case.

When $n=2$, we have $k=1$ and $\max\{0,a-1\}<s\le\min\{a,1\}$. Therefore $\lceil s\rceil=1$, $\floor{a-s} =0$ and $S_{n, k}(a,s)=\max\{0,2s-a\}$. Kaufman \cite{Kaufman69} showed that if $2s-a\ge 0$, then indeed $T_{n,1}(a,s)\ge 2s-a$. On the other hand, if $2s-a<0$ then the lower bound is trivial.

Now let $n\ge 3$ be arbitrary. Our proof is based on the following, easy observations:  
   \begin{itemize}
   \item[(i)] If $A\subset \R^{n-1}$, $A'=A\times\{0\}\subset \R^{n}$, $V\in G(n-1,k)$, $V'\in G(n,k)$ and $\pi_{\R^{n-1}}V'=V$ (this is projection onto the first
$(n - 1)$ coordinates), then $\dim_H(\pi_V A)=\dim_H(\pi_{V'} A')$. Consequently, $T_{n,k}(a,s)\ge T_{n-1,k}(a,s)+k$. Note that the added $k$ is because there are $k$ degrees of freedom when picking a $V'$ which projects onto a given $V$. Namely, to form $V'$ we can pick a basis for $V$ (as a subspace of $\R^{n - 1}$) and then fill in whatever we want in the $n^\text{th}$ coordinate of each vector in this basis.
% given a basis of $V'$, there are $k$ degrees
% of freedom, one for each $n$th coordinate of each basis vector). 
\item[(ii)] If $A\subset \R^{n-1}$, $A'=A\times\R\subset \R^{n}$, $V\in G(n-1,k)$, $V'=V\times\{0\}\in G(n,k)$, then $\dim_H(\pi_V A)=\dim_H(\pi_{V'} A')$. Consequently, $T_{n,k}(a,s)\ge T_{n-1,k}(a-1,s)$.
\item[(iii)] If $A\subset \R^{n-1}$, $A'=A\times\{0\}\subset \R^{n}$, $V\in G(n-1,k-1)$, $V'=V\times\R\in G(n,k)$, then $\dim_H(\pi_V A)=\dim_H(\pi_{V'} A')$.  Consequently, $T_{n,k}(a,s)\ge T_{n-1,k-1}(a,s)$.
\item[(iv)] If $A\subset \R^{n-1}$, $A'=A\times\R\subset \R^{n}$, $V\in G(n-1,k-1)$, $V'\in G(n,k)$, $V'\supset V$, then $\dim_H(\pi_V A)+1 \geq \dim_H(\pi_{V'} A')$. Consequently, $T_{n,k}(a,s)\ge T_{n-1,k-1}(a-1,s-1)+n-k$. Note that the $(n - k)$ term appears because there are $n-k$ degrees of freedom when picking $V'$ which extends $V$. In particular, to extend $V'$ to $V$, we can pick one unit vector from the $n - (k - 1)$-dimensional subspace $V^\perp$ (where $V$ is considered as a subspace of $\R^n$ rather than $\R^{n - 1}$).
\end{itemize}

Assume that $n,k,a,s$ satisfy the conditions of Theorem \ref{thm:maintheorem}. We distinguish four cases: If $k, a, s> 1$, then $n-1,k-1,a-1,s-1$ also satisfy the conditions of Theorem \ref{thm:maintheorem}, and hence, by (iv),
 $$T_{n,k}(a,s)\ge T_{n-1,k-1}(a-1,s-1)+n-k\ge S_{n-1,k-1}(a-1,s-1)+n-k=S_{n,k}(a,s).$$

If $s \leq 1$, $a < 1+s$, and $k \leq n -2$, then $n-1,k,a,s$ satisfy the conditions of Theorem \ref{thm:maintheorem}, and hence, by (i), $$T_{n,k}(a,s) \ge T_{n-1,k}(a,s) + k \ge S_{n-1,k}(a,s)+k = S_{n,k}(a,s).$$

If $s \leq 1$, $a \geq 1+s$, and $k \leq n -2$, then $n-1,k,a-1,s$ satisfy the conditions of Theorem \ref{thm:maintheorem}, and hence, by (ii), $$T_{n,k}(a,s) \ge T_{n-1,k}(a-1,s) \ge S_{n-1,k}(a-1,s) = S_{n,k}(a,s).$$

Finally, if $s \leq 1$ and $k = n-1$, then $n-1,k-1,a,s$ satisfy the conditions of Theorem \ref{thm:maintheorem}, and hence, by (iii), $$T_{n,n-1}(a,s) \ge T_{n-1,n-2}(a,s) \ge S_{n-1,n-2}(a,s) = S_{n,n-1}(a,s).\qed $$

\section{Proof of upper bounds}

In this section, we will prove the second part of Theorem~\ref{thm:maintheorem}, namely that when $k = 1$ or $k = n - 1$, $T_{n, k}(a, s) \leq S_{n, k}(a, s)$. Our proof can be thought of as a `reverse' of the proof of the lower bound discussed in the previous section. In order to explain what we mean by this, it helps to introduce the following definition.
\begin{definition}
We say that sets $A\subseteq\R^n$, $E\subseteq G(n,k)$ are an $(n, k, a, s, t)$-pair if $A$ is analytic, both $A$ and $E$ are nonempty, $\dim_H(A) = a$, $\dim_H(E) = t$, and $\dim_H(\pi_V A) < s$ for every $V \in E$.
\end{definition}

\begin{remark}
Note that $T_{n, k}(a, s)$ is by definition the supremum of $t$ over all $(n, k, a, s, t)$-pairs.
\end{remark}

The key idea of our proof is that given an $(n + 1, k, a, s, t)$-pair, we can use it to construct an $(n, k', a', s', t')$-pair for some appropriately chosen $k', a', s'$ and $t'$. This allows us to lift upper bounds on $T_{n,k}(a, s)$ from $\R^n$ to $\R^{n + 1}$ and thus prove our upper bound by induction on $n$ (with Ren and Wang's Theorem~\ref{thm:RenWang} providing the base case).

Our proof of the lower bound on $T_{n, k}(a, s)$ from the previous section can be seen as dual to this proof. In particular, observations $(i)$ through $(iv)$ in that proof can all be seen as ways of taking an $(n, k, a, s, t)$-pair and using it to construct an $(n + 1, k', a', s', t')$-pair for some appropriately chosen $k', a', s'$ and $t'$. Moreover, these constructions closely mirror the constructions we will use in our proof of the upper bound---in a sense, the constructions in the proof of the upper bound can be seen as ways to ``undo'' the constructions from the proof of the lower bound.

The following four lemmas state precisely the different sorts of ways we have of turning an $(n + 1, k, a, s, t)$-pair into an $(n, k', a', s', t')$-pair. The first pair of lemmas will be used to prove the upper bound on $T_{n, k}(a, s)$ when $k = 1$ and are dual to observations $(i)$ and $(ii)$, respectively, from the previous section.

\begin{restatable}{lemma}{uppertwo}
\label{lem:upper2}
Let $n\geq 2$. Suppose that $0 < a \leq n$, $a-n < s\leq \min\{a, 1\}$, $s >0$, and $t > 1$. If there is an $(n+1, 1, a, s, t)$-pair, then there is an $(n,1,a,s,t^\prime-1)$-pair for every $1 < t^\prime < t$.
\end{restatable}

\begin{restatable}{lemma}{upperone}
\label{lem:upper1}
Let $n\geq 2$. Suppose that $1 < a \leq n+1$, $a-n < s\leq \min\{a-1, 1\}$, and $s > 0$. If there is an $(n+1, 1, a, s, t)$-pair, then there is an $(n, 1, a-1, s, t^\prime)$-pair, where $t^\prime = \min\{t, n-1\}$.
\end{restatable}

The second pair of lemmas will be used to prove the upper bound on $T_{n, k}(a, s)$ when $k = n - 1$ and are dual to observations $(iii)$ and $(iv)$, respectively, from the previous section.

\begin{restatable}{lemma}{upperthree}
\label{lem:upper3}
    If there is an $(n+1, n, a, s, t)$-pair, with $s \leq n-1$ and $s > a-1$, then there is an $(n, n-1, a, s, t^\prime)$-pair, where $t^\prime = \min\{t, n-1\}$.
\end{restatable}

\begin{restatable}{lemma}{upperfour}
\label{lem:upper4}
    If there is an $(n+1,n, a, s, t)$-pair, with $a,s,t > 1$, then there is an $(n, n-1, a-1, s-1, t^\prime-1)$-pair for all $1< t^\prime<t$.
\end{restatable}

Below, we will prove these four lemmas and then use them to complete the inductive arguments proving the upper bounds on $T_{n, k}(a, s)$. But first we will review some standard projection and slicing results that will be needed for the proofs of the lemmas.

\begin{remark}
\label{rmk:analytic}
In order to apply the projection and slicing results in the next subsection, we will often need to assume that for a given $(n, k, a, s, t)$-pair $(A, E)$, the set $E$ of exceptional subspaces is analytic. We will now explain why it is always safe to make this assumption. Since $A$ is analytic, it has an $F_\sigma$ subset of the same dimension. Moreover, restricting to such a subset cannot shrink the set of exceptional subspaces and hence we may always assume that $A$ is $F_\sigma$. It is then relatively straightforward to show that the set of exceptional subspaces of $A$, namely $\{V \in G(n, k) \mid \dim_H(\pi_V(A)) < s\}$, is analytic. For the rest of this paper, we will assume without further comment that all $(n, k, a, s, t)$-pairs $(A, E)$ have both $A$ and $E$ analytic.
\end{remark}

\subsection{Projection and slicing results}
 Let $V\in G(n, n-1)$, and let $G(V,1)$ be the set of one-dimensional linear subspaces of $V$. We denote the orthogonal projection of the lines in $G(n,1)\setminus\{V^\perp\}$ onto $V$ by $$\pi_V: G(n,1)\setminus\{V^\perp\} \rightarrow G(V,1).$$

As discussed in the Introduction, Mattila \cite{Mattila75} generalized Marstrand's projection theorem to higher dimensions. We will need a simple corollary of his result in the context of orthogonal projections of subspaces.
\begin{lemma}\label{lem:projectionSubspaces}
    Let $E\subseteq G(n, 1)$. For almost every $V\in G(n, n-1)$, $$\dim_H(\pi_V (E\setminus\{V^\perp\})) = \min\{\dim_H(E), n-2\}.$$
\end{lemma}
\begin{proof}
    Let $E^*\subset\R^n$ be the set of all points covered by the lines in $E$, i.e., $E^*$ is a union of lines through the origin. Note that $\dim_H(E^*) = \dim_H(E)+1$. By Mattila's projection theorem, for almost every $V$, $\dim_H(\pi_V E^*) = \min\{\dim_H(E^*), n-1\}$. Therefore, for almost every $V\in G(n,n-1)$,
$$\dim_H(\pi_V (E\setminus\{V^\perp\})) = \dim_H(\pi_V (E\setminus \{V^\perp\})^*) -1=\dim_H(\pi_V E^*) -1= \min\{\dim_H(E), n-2\},$$
    as claimed.
  \end{proof}

Mattila \cite{Mattila75} also proved the following slicing theorem.
\begin{theorem}\label{thm:mattilaslicing}
    Let $n - k \leq s \leq n$, and let $A\subseteq\R^n$ be $\mathcal{H}^s$-measurable with $0 < \mathcal{H}^s(A)< \infty$. Then, for almost all $V\in G(n,k)$,
    \begin{equation*}
        \mathcal{H}^{n-k}(\{u\in V^\perp\mid \dim_H(A\cap (V+u)) = s+ k -n\}) > 0.
    \end{equation*}
\end{theorem}

We will need the following corollary of Mattila's slicing theorem.
\begin{corollary}\label{cor:slicingSn}
    Let $t>1$, and let $E\subset G(n,1)$ be a compact set with $0<\mathcal{H}^t(E) < \infty$. Then, for positively many $V\in G(n, n-1)$, 
    \begin{equation}\label{eq:slicingSn}
        \dim_H(\{W\in E\mid W\subset V\}) = t-1.
    \end{equation}
\end{corollary}
\begin{proof}

Write $\R^n=\R^{n-1}\times\R$. Without loss of generality we can assume that the dimension of the set of lines in $E$ not contained in the `horizontal' hyperplane $\R^{n-1}$ is $t$. As before, let $E^*\subset\R^n$ be the set of points covered by the lines in $E$. Let $H$ denote the horizontal affine hyperplane $H=\R^{n-1}\times\{1\}$. Let $A = E^* \cap H$,
    i.e., $A$ is the set of points at which the lines in $E$ intersect $H$ (sometimes also called the gnomonic projection of $E$). Note that $\dim_H(A) = \dim_H(E) = t$. 

By Theorem \ref{thm:mattilaslicing}, for almost every $V''\in G(n-1, n-2)$, there are positively many $(n-2)$-dimensional affine subspaces $V'\subset H$ parallel to $V''$ for which $\dim_H(A\cap V')= t-1$. It is immediate to see that, for any such $V'$, \eqref{eq:slicingSn} holds with $V=\mathrm{span}\,{V'}\in G(n,n-1)$. 

To finish the proof, we note that if $V_1'', V_1'$ and $V_2'', V_2'$ are both as in the paragraph above (i.e. $V_1'' \in G(n - 1, n - 2)$ and is a hyperplane in $H$ and $V_1' \subseteq H$ is an $(n - 2)$-dimensional affine subspace of $H$ parallel to $V_1''$, and similarly for $V_2'', V_2'$) and if $(V_1'', V_1') \neq (V_2'', V_2')$ then $\mathrm{span}(V_1') \neq \mathrm{span}(V_2')$. Hence we can apply Fubini's theorem to conclude that the set $\{V \in G(n, n - 1) \mid \dim_H(\{W \in E \mid W \subset V\}) = t - 1\}$ has positive measure. In order to see that we can apply Fubini's theorem here, note that it follows from the fact that $E$ is compact that this set is Borel measurable.
\end{proof}

\subsection{Proofs of Lemmas \ref{lem:upper2}--\ref{lem:upper4}}

In the first two proofs below, suppose that $A\subseteq\R^{n+1}$ and $E\subseteq G(n+1, 1)$ form an $(n+1, 1, a, s,t)$-pair. As discussed in Remark~\ref{rmk:analytic}, we may assume that $E$ is analytic.

\uppertwo*

\begin{proof}%[Proof of Lemma  \ref{lem:upper2}]
Let $1 < t^\prime < t$. Note that since the dimension of $E$ is strictly greater than $t^\prime$, by passing to a subset of $E$, we may assume that $E$ is compact and $0 < \mathcal{H}^{t^\prime}(E) < \infty$. 

Let $V\in G(n+1, n)$ be any subspace such that $\dim_H(\pi_V A) = a$ and $\dim_H(E')=t^\prime-1$ for $E':=\{W\in E\mid W\subset V\}.$ Note that, by Mattila's projection theorem and Corollary \ref{cor:slicingSn}, such a $V$ must exist. Write $\R^{n+1}=\R^n\times\R$ and identify $V$ with $\R^n$. Consider $E'$ as a set in $G(n, 1)$. Then, since $\pi_W A = \pi_{W}A^\prime$ for $A':=\pi_V A$ and for every $W\in E'$, it follows that $(A^\prime, E^\prime)$ 
is an $(n, 1, a, s, t^\prime-1)$ pair.
\end{proof}

\upperone*

\begin{proof}%[Proof of Lemma  \ref{lem:upper1}]
By Mattila's slicing theorem,  Theorem \ref{thm:mattilaslicing}, 
%EM added theorem number
for almost every $V \in G(n+1, n)$, there is a $y\in V^\perp$ such that $A^\prime := A \cap (V + y)$ has dimension $a-1$. And by Lemma \ref{lem:projectionSubspaces}, for almost every $V\in G(n+1, n)$, with $E^\prime = \pi_V(E \setminus \{V^\perp\})$, $\dim_H(E^\prime) = \min\{t,n-1\}$. We therefore fix a $V$ satisfying both properties. Write $\R^{n + 1} = \R^n \times \R$ and identify $V$ with $\R^n$.

To complete the proof, we show that $(A',E')$ is an $(n,1,a-1,s,t')$-pair, i.e., that $\dim_H(\pi_{W} A^\prime) < s$ for all $W \in E^\prime$. Let $W^*\subset\R^n$ be the set of points covered by $W$, and let $E^*\subset \R^{n+1}$ be the set of points covered by the lines of $E$. We need to show that for every $u_0\in W^*$, $\dim_H(u_0\cdot A')<s$.

Note that $A^\prime = \{x\in\R^n \mid (x, y) \in A\}$, and if $u_0\in W^*$ then there is some $u_1$ such that $u = (u_0, u_1) \in E^*$. Therefore $u\cdot (A \cap (V+y)) = \{(u_0, u_1)\cdot (x, y) \mid (x, y) \in A\} = u_0\cdot A^\prime + u_1\cdot y$ is a translate of $u_0\cdot A'$ and thus $\dim_H(u_0\cdot A') = \dim_H(u\cdot (A\cap (V +y))) \leq \dim_H(u\cdot A) < s$. 
\end{proof}

In the next two proofs, suppose that $A\subseteq \R^{n + 1}$ and $E \subseteq G(n + 1, n)$ form an $(n + 1, n, a, s, t)$-pair. Once again, we may assume $E$ is analytic.

\upperthree*

\begin{proof}[Proof of Lemma \ref{lem:upper3}]
We first note that our assumptions on $s$ and $a$ imply that $a \leq n$.
 By Mattila's projection theorem, $A^\prime := \pi_V A$ has dimension $a$ for almost every $V \in G(n+1, n)$. Moreover, by Lemma \ref{lem:projectionSubspaces} applied to the perpendiculars of $E$, for almost every $V \in G(n+1, n)$, the set $E^\prime := \{V\cap W \mid W \in E\} \subseteq G(V, n-1)$ satisfies $\dim_H(E^\prime) = \min\{t, n-1\}$. We therefore fix a $V\in G(n+1,n)$ satisfying both properties, and as usual, identify $V$ with $\R^n$.
 %without loss of generality we assume that $V=\R^n$ is the horizontal hyperplane. 
 Considering $E'$ as a set in $G(n,n-1)$, since  $\dim_H(\pi_{W\cap V} A^\prime) \leq \dim_H(\pi_W A)$ for $A'= \pi_V A$ and for every $W\in E$, therefore $(A',E')$ is an $(n,n-1,a,s,t')$ pair.
\end{proof}

\upperfour*

\begin{proof}[Proof of Lemma \ref{lem:upper4}]
Fix $1 < t^\prime < t$ and let $t^\prime < t'' < t$. Note that since the dimension of $E$ is strictly greater than $t''$, by passing to a subset of $E$ we may assume that $E$ is compact and that $0 < \mathcal{H}^{t''}(E) < \infty$. By Corollary \ref{cor:slicingSn}, there are positively many $V\in G(n+1, n)$ such that 
\begin{center}
    $E^\prime := \{W \in E \mid W^\perp \subset V\}$
\end{center}
has dimension $t''-1$. By Mattila's slicing theorem, Theorem \ref{thm:mattilaslicing}, for almost every $V\in G(n+1,n)$, for positively many $x \in V^\perp$, the set $A_x = A \cap (V + x)$ satisfies
\begin{equation*}\label{eq:dimensionAx}
    \dim_H(A_x) = a -1.
\end{equation*}
Thus we may fix a hyperplane $V\in G(n+1,n)$ satisfying both properties.

Let $\mu$ be the restriction of $\mathcal{H}^{t' - 1}$ onto an analytic subset of $E^\prime$ of positive and finite $\mathcal{H}^{t' - 1}$  measure (this is possible because $t' - 1$ is strictly less than the dimension of $E^\prime$, which is $t'' - 1$). Note that by Marstrand's vertical slicing theorem \cite{Marstrand54}, generalized to arbitrary dimension by Mattila \cite{Mattila75}, for every $W\in E'$, the set $\pi_W A_x=(\pi_WA)\cap (V+x)$ has dimension at most $\dim_H(\pi_W A)-1 < s - 1$ for a.e. $x\in V^\perp$. Hence by Fubini's theorem (applied to the product of $\mu$ and the Lebesgue measure on $V^\perp$), for almost every $x \in V^\perp$, 
\begin{equation}\label{eq:dimensionSliceProjection}
    \dim_H(\pi_{W} A_x) \leq s-1
\end{equation}
holds for $\mu$-almost every $W \in E'$. Fix one such $x$ and let $E''$ be an analytic subset of $E'$ of dimension $t' - 1$ such that for every $W \in E''$, \eqref{eq:dimensionSliceProjection} holds. Then $A_x$ and $E''$ form an $(n, n - 1, a - 1, s - 1, t')$-pair and the conclusion follows.
    % By Marstrand's vertical slicing theorem \cite{Marstrand54}, generalized to arbitrary dimension by Mattila \cite{Mattila75}, for any $W\in E'$, the set $\pi_W A_x=(\pi_WA)\cap (V+x)$ has dimension at most $\dim_H(\pi_W A)-1 < s - 1$ for a.e. $x\in V^\perp$. 
    % Hence, by Fubini's theorem, for every $\sigma$-finite measure $\mu$ on $E'$, for almost every $x\in V^\perp$,
    % \begin{equation}\label{eq:dimensionSliceProjection}
    %     \dim_H(\pi_{W} A_x) \leq s-1
    % \end{equation}
    % holds for $\mu$-a.e. $W\in E'$. Let $0<t'' < t^\prime-1$. Let $\mu$ be the restriction of $\mathcal{H}^{t''}$ onto an analytic subset of $E^\prime$ of positive and finite $\mathcal{H}^{t''}$  measure. Since \eqref{eq:dimensionSliceProjection} holds for $\mu$-a.e. $W \in E^\prime$, there is a $t''$ dimensional analytic subset, $E^\prime(t'')$, such that \eqref{eq:dimensionSliceProjection} holds for every $W\in E^\prime(t'')$. Therefore $(A_x,E^\prime(t''))$ is an $(n, n-1, a-1,s-1,t'')$-pair, and the conclusion follows.
\end{proof}

\subsection{Proof of upper bound when $k = 1$}
  Recall that if $k=1$, our aim is to prove that for any $n\geq 2$, $0<a\leq n$ and $\max\{0,a-n+1\}<s\leq \min\{a,1\}$,
  %EM missing \}
      \begin{center}
        $T_{n,1}(a,s) \le S_{n,1}(a,s) = n - 2 -\floor{a-s} + \max\{0, \floor{a-s}+2s-a\}$.
    \end{center}
We prove this by induction on $n$. The base case, $n = 2$, follows immediately from Theorem \ref{thm:RenWang}. We now assume that it holds up to $n$ and prove that it also holds for $n + 1$.

Let $A\subseteq\R^{n+1}$ be an analytic set with $\dim_H(A) = a$, $a - n < s \leq \min\{a, 1\}$, and let $E = \{V\in G(n+1,1)\mid \dim_H(\pi_V A) < s\}$. Let $t= \dim_H(E)$. That is, $(A, E)$ is an $(n+1, 1, a, s, t)$-pair.\\

\noindent \textbf{Case 1:} Assume that 
    either $s < a-1$, or $s = a-1$ and $a < 2$.
    %EM ,
    By Lemma \ref{lem:upper1}, there is an $(n, 1, a-1, s, t^\prime)$-pair,
    where 
 %   \begin{center}
        $t^\prime := \min\{t, n-1\}$.
 %   \end{center}
    By our induction hypothesis,
    \begin{align*}
        t^\prime &\leq S_{n,1}(a-1,s)= S_{n+1,1}(a, s).
    \end{align*}
    
    We first show that the case $t^\prime = n-1$ leads to a contradiction. To see this, assume $t^\prime = n-1$, we will then have
    \begin{equation}\label{eq:thmlines1}
        \floor{a-s}\leq \max\{0, \floor{a-s} + 2s-a\}.
    \end{equation}
    Since $\floor{a-s} \geq 1$, therefore $a\leq 2s$, and so $a \leq 2$. If $s < a-1$, then $a > 2$, i.e., we have a contradiction. If $s = a-1$ and $a < 2$, we again have a contradiction since \eqref{eq:thmlines1} fails. 
    
    Therefore, $t^\prime = t$, and the conclusion follows. \\

    \noindent \textbf{Case 2:} Assume that $s = 1$ 
    %s = a-1
    and $a= 2$. It suffices to show that
    % \begin{equation}
        $t \leq n-1$.
    % \end{equation}
    This follows from Falconer's bound, \eqref{eq:FalconerBound}. \\

    \noindent \textbf{Case 3:} Assume that $s > a-1$. If $t \leq 1$, then we are done, since $1 \leq S_{n+1,1}(a, s)$ when $s >a-1$. Otherwise, by Lemma \ref{lem:upper2}, there is an $(n, 1, a, s, t^\prime-1)$-pair for every $1<t^\prime <t$. By our induction hypothesis,
   \begin{align*}
        t^\prime-1 &\leq S_{n,1}(a,s)\\
        &= S_{n+1,1}(a,s) -1.
    \end{align*}
    From the definition of $T_{n,k}$, letting $t^\prime$ go to $t$ concludes the proof.
%\end{proof}

\subsection{Proof of upper bound when $k=n-1$}
Our aim is to prove that if $n\geq 2$, $0 < a \leq n$ and $a - 1 < s \leq \min\{a, n-1\}$, then
\begin{equation*}
    T_{n, n-1}(a, s)\le S_{n,n-1}(a,s) = \ceil{s}-1+ \max\{0, 2s-a -\ceil{s} +1\}.
\end{equation*}

    We prove this by induction on $n$. The base case, $n = 2$ is due to Ren and Wang \cite{ren2023}. We assume that it holds up to $n$. 

    Let $(A,E)$ be an $(n+1,n,a,s,t)$ pair. \\

    \noindent \textbf{Case 1:} Assume that 
    either $s \leq n -1$ and $s < a$, or $s < n -1$ and $s \leq a$. By Lemma \ref{lem:upper3}, there is an $(n, n-1, a, s, t^\prime)$-pair, where $t^\prime = \min\{t,n-1\}$. By our induction hypothesis, 
  %  \begin{align*}
    $$    t^\prime \leq S_{n,n-1}(a,s)= S_{n+1, n}(a, s).$$
   % \end{align*}
    Suppose that $t^\prime = n-1$. If $S_{n,n-1}(a,s)=\ceil{s}-1$, we have an immediate contradiction. Therefore, $S_{n,n-1}(a,s) = 2s-a$, and so $\frac{a + n-1}{2}\leq s$. If  $s\leq n-1$  and $s < a$, we must have $a > n -1$, but this is a contradiction. Similarly, if $s < n-1$ and $s \leq a$, we must have $a< n-1$, 
    % EM changed \leq to <
    which is a contradiction. Therefore $t^\prime = t$, and the conclusion follows.\\

    \noindent \textbf{Case 2:} Assume that $s = n-1 = a$. It suffices to show that $t\leq S_{n+1,n,}(n-1,n-1) = n -1$. This follows directly from Mattila's bound \eqref{eq:MattilaBound}.\\

    \noindent \textbf{Case 3:} Assume that $s > n-1$. If $t \leq 1$, then $t \leq S_{n+1,n}(a,s)$ trivially, so we may assume that $t > 1$. By Lemma \ref{lem:upper4}, for every $1<t^\prime < t$ there is an $(n, n-1, a-1, s-1, t^\prime-1)$-pair. Therefore, by our induction hypothesis,
    \begin{align*}
        t^\prime - 1 &\leq S_{n, n-1}(a-1, s-1)= S_{n+1,n}(a,s) - 1.
    \end{align*}
    From the definition of $T_{n,k}$, letting $t^\prime$ go to $t$ concludes the proof. \qed
%\end{proof}

\bibliographystyle{plain}
\bibliography{SQuaREs.bib}

\end{document}